\documentclass[12pt]{article}

\usepackage{calrsfs}
\usepackage{amsthm,amsfonts,amsbsy,amssymb,amsmath,amscd}
\usepackage[all]{xy}
\usepackage{xr-hyper}
\usepackage[colorlinks=true]{hyperref}
\usepackage{fullpage}
\usepackage{verbatim}
\usepackage[dvips]{lscape}
\usepackage{hyperref} %label of reference

\begin{document}
\title{On Volumes of Arithmetic Line Bundles}
\author{Xinyi Yuan
\footnote{The author is fully supported by a research fellowship of the Clay Mathematics Institute.}
}
\maketitle

\theoremstyle{plain}
\newtheorem{thm}{Theorem}[section]
\newtheorem*{conj}{Conjecture}
\newtheorem*{notation}{Notation}
\newtheorem{cor}[thm]{Corollary}
\newtheorem*{corr}{Corollary}
\newtheorem{lem}[thm]{Lemma}
\newtheorem{pro}[thm]{Proposition}
\newtheorem{definition}[thm]{Definition}
\newtheorem*{thmm}{Theorem}

\theoremstyle{remark} \newtheorem*{remark}{Remark}
\theoremstyle{remark} \newtheorem*{example}{Example}

\newcommand{\RR}{\mathbb{R}}
\newcommand{\QQ}{\mathbb{Q}}
\newcommand{\CC}{\mathbb{C}}
\newcommand{\ZZ}{\mathbb{Z}}
\newcommand{\FF}{\mathbb{F}}
\newcommand{\PP}{\mathbb{P}} 

\newcommand{\fp}{\mathbb{F}_p}
\newcommand{\fwp}{\mathbb{F}_{\wp}}

\newcommand{\rank}{\mathrm{rank}}        % rank
\newcommand{\vol}{\mathrm{vol}}          % volume
\newcommand{\divv}{\mathrm{div}}         % divisor
\newcommand{\Spec}{\mathrm{Spec}}        % Spec
\newcommand{\ord}{\mathrm{ord}}          % ord

\newcommand{\lb}{\mathcal{L}}            % Line bundle
\newcommand{\mb}{\mathcal{M}}
\newcommand{\nb}{\mathcal{N}}
\newcommand{\fb}{\mathcal{F}}
\newcommand{\eb}{\mathcal{E}}
\newcommand{\tb}{\mathcal{T}}
\newcommand{\ob}{\mathcal{O}}
\newcommand{\ib}{\mathcal{I}}
\newcommand{\jb}{\mathcal{J}}
\newcommand{\ab}{\mathcal{A}}
\newcommand{\bb}{\mathcal{B}} 

\newcommand{\hhat}{\hat h^0}
\newcommand{\Hhat}{\widehat H^0}

\newcommand{\xb}{\mathcal{X}}             % integral model

\newcommand{\lbb}{\overline{\mathcal{L}}}             % Line bundles with bar
\newcommand{\mbb}{\overline{\mathcal{M}}}
\newcommand{\nbb}{\overline{\mathcal{N}}}
\newcommand{\ebb}{\overline{\mathcal{E}}}
\newcommand{\tbb}{\overline{\mathcal{T}}}
\newcommand{\obb}{\overline{\mathcal{O}}}
\newcommand{\abb}{\overline{\mathcal{A}}}
\newcommand{\bbb}{\overline{\mathcal{B}}}

\newcommand{\lbt}{\widetilde{\mathcal{L}}}
\newcommand{\mbt}{\widetilde{\mathcal{M}}}

\newcommand{\chern}{\hat{c}_1}           % Chern classes with hat
\newcommand{\height}{h_{\lbb}}           % height
\newcommand{\cheight}{\hat{h}_{\lb}}     % canonical height

\newcommand{\lnorm}[1]{\|#1\|_{L^2}}     % L_2 norm
\newcommand{\supnorm}[1]{\|#1\|_{\mathrm{sup}}}    % sup norm
\newcommand{\chil}{\chi_{_{L^2}}}
\newcommand{\chisup}{\chi_{\sup}}

\newcommand{\amp}{\widehat{\mathrm{Amp}}(X)}   % Ample(X)   
\newcommand{\nef}{\widehat{\mathrm{Nef}}(X)}   
\newcommand{\eff}{\widehat{\mathrm{Eff}}(X)}   
\newcommand{\pic}{\widehat{\mathrm{Pic}}(X)}   
\newcommand{\bigx}{\widehat{\mathrm{Big}}(X)}

\tableofcontents

\section{Introduction}
In their recent paper \cite{LM}, Lazarsfeld and Musta\c t\v a explored a systematic way of using Okounkov bodies, originated in \cite{Ok1, Ok2}, to study the volumes of line bundles over algebraic varieties. They easily recovered many positivity results in algebraic geometry (cf. \cite{La}). A similar construction with a different viewpoint was also taken by Kaveh-Khovanskii \cite{KK}. Our paper is the expected arithmetic analogue of \cite{LM}. Our main results are as follows:

\begin{itemize}
\item Introduce arithmetic Okounkov bodies associated to an arithmetic line bundle, and prove that the volumes of the former approximate the volume of the later; 
\item Show some log-concavity inequalities on the volumes and top intersection numbers, which can be viewed as a high-dimensional generalization of the Hodge index theorem on arithmetic surfaces of Faltings \cite{Fa}; 
\item Prove an arithmetic analogue of Fujita's approximation theorem, which is proved independently by Huayi Chen \cite{Ch2} during the preparation of this paper;
\item As by-products, we recover the convergence of $\displaystyle\frac{\hhat(X,m\lbb)}{m^d/d!}$ proved by Chen \cite{Ch1} and the arithmetic Hodge index theorem in codimension one proved by Moriwaki \cite{Mo1}.
\end{itemize}

\subsection*{Volume of an arithmetic line bundle}
Let $X$ be an arithmetic variety of dimension $d$. That is, $X$ is a $d$-dimensional integral scheme, projective and flat over $\Spec(\ZZ)$.
For any Hermitian line bundle $\lbb=(\lb, \|\cdot\|)$ over $X$,
denote 
$$\Hhat(X,\lbb)= \{s\in H^0(X,\lb): \|s\|_{\sup} \leq 1\}$$ 
and 
$$\hhat(X,\lbb)= \log \# \Hhat(X,\lbb).$$
Define the volume to be
$$\vol(\lbb)=\limsup_{m\rightarrow \infty}  \frac{\hhat(X,m\lbb)}{m^d/d!}.$$
Note that in this paper we write line bundles additively, so $m\lbb$ means $\lbb^{\otimes m}$.

A line bundle $\lbb$ is said to be big if $\vol(\lbb)>0$; it is effective if $\hhat(X,\lbb)>0$.
If $\lbb$ is ample in the sense of 
Zhang \cite{Zh} (cf. Section \ref{section ample and big}), then it is big by the formula
$$\vol(\lbb)=\lbb^d>0.$$
This is a result combining the works of Gillet-Soul\'e \cite{GS3, GS2}, 
Bismut-Vasserot \cite{BV} and Zhang \cite{Zh}. See \cite[Corollary 2.7]{Yu} for example.

As pointed out above, Huayi Chen proved that ``limsup=lim'' in the definition of $\vol(\lbb)$ in his recent work \cite{Ch1} using Harder-Narasimhan filtrations. We will derive this result by means of Okounkov bodies in Theorem \ref{limit}. 

Some basic properties of big line bundles are proved in \cite{Mo2, Mo3} and \cite{Yu}. They use different definitions of bigness, but \cite[Corollary 2.4]{Yu} shows that all of these definitions are equivalent.
The result ``big=ample+effective" of \cite{Yu} rephrased in Theorem \ref{ample+effective} will be widely used in this paper.
It is also worth noting that the main result of Moriwaki \cite{Mo3} asserts that the volume function is continuous at all hermitian line bundles.

\subsection*{Okounkov Body}
Assume that $X$ is normal with smooth generic fibre $X_{\QQ}$. Let 
$$X \supset Y_1 \supset \cdots \supset Y_d$$
be a flag on $X$, where each $Y_i$ is a regular irreducible closed subscheme of codimension $i$ in $X$. 
We require that $Y_1$ is a vertical divisor lying over some finite prime $p$, and write $\mathrm{char}(Y.)=p$ in this case. We further require that the residue field 
of $Y_d$ is isomorphic to $\fp$.
There is a positive density of such prime $p$ for which $Y.$ exists.

Define a valuation map 
$$\nu_{Y.}=(\nu_1, \cdots, \nu_d): H^0(X,\lb)-\{0\} \rightarrow \ZZ^d$$
with respect to the flag $Y.$ as in \cite{LM}. We explain it here. 
For any nonzero $s\in H^0(X,\lb)$, we first set $\nu_1(s)=\ord_{Y_1}(s)$. 
Let $s_{Y_1}$ be a section of the line bundle $\ob(Y_1)$ with zero locus $Y_1$. Then 
$s_{Y_1}^{\otimes(-\nu_1(s))}s$ is nonzero on $Y_1$, and let $s_1=\left. \left(s_{Y_1}^{\otimes(-\nu_1(s))}s\right)\right|_{Y_1}$ be the restriction. 
Set $\nu_2(s)=\ord_{Y_2}(s_1)$. 
Continue this process on the section $s_1$ on $Y_2$, we can define $\nu_3(s)$ and thus $\nu_4(s), \cdots, \nu_d(s)$.

For any Hermitian line bundle $\lbb$ on $X$, denote 
$$v_{Y.}(\lbb)=\nu_{Y.}(\Hhat(X,\lbb)-\{0\})$$
to be the image in $\ZZ^d$.
\textbf{Note that we only pick up the image of the finite set $\Hhat(X,\lbb)-\{0\}$.}

Let $\Delta_{Y.}(\lbb)$ be the closure of 
$\Lambda_{Y.}(\lbb)=\displaystyle\bigcup_{m\geq 1}\frac{1}{m^d} v_{Y.}(m\lbb)$
in $\RR^d$. It turns out that $\Delta_{Y.}(\lbb)$ is a bounded convex subset of $\RR^d$ if non-empty. It has a finite volume under the Lebesgue measure of $\RR^d$.
See Lemma \ref{convex and bounded}. 
We have the following counterpart of \cite[Theorem A]{LM}.

\theoremstyle{plain}\newtheorem*{thma}{Theorem A}
\begin{thma}
If $\lbb$ is big, then
$$\lim_{p=\mathrm{char}(Y.)\rightarrow \infty} 
\vol(\Delta_{Y.}(\lbb)) \log p
=\frac{1}{d!}\vol(\lbb). $$
\end{thma}

In the geometric case of \cite{LM}, exact equality between two volumes are easily obtained without taking limit on $p$. But it seems hard to be true in the arithmetic case if $\lbb$ is not ample. However, the above result is apparently sufficient for applications of $\vol(\lbb)$.

\subsection*{Log-concavity and Hodge index theorem}
As what Lazarsfeld and Musta\c t\v a do, we also show the log-concavity of volume functions by the classical Brunn-Minkowski theorem in Euclidean geometry. 

\theoremstyle{plain}\newtheorem*{thmb}{Theorem B}
\begin{thmb}
For any two effective line bundles $\lbb_1, \lbb_2$, we have
$$\vol(\lbb_1+\lbb_2)^{\frac 1d}\geq \vol(\lbb_1)^{\frac 1d}+ \vol(\lbb_2)^{\frac 1d}.$$
\end{thmb}

When $\lbb_1$ and $\lbb_2$ are ample, the above volumes are exactly equal to arithmetic intersection numbers. Even in this case, the inequality is not as transparent as the geometric case. The result can be viewed as a generalization of the Hodge index theorem on arithmetic surfaces. See \cite{La} for many related inequalities in the geometric case. 

An easy consequence of the above result is the relation
\begin{align*}
 (\lbb_1^{d-1} \cdot \lbb_2 )^2  
\geq & (\lbb_1^d  )  ( \lbb_1^{d-2} \cdot \lbb_2^2)
\end{align*}
on intersection numbers for any two ample line bundles $\lbb_1, \lbb_2$. 
This simple-looking relation is equivalent to the Hodge index theorem for divisors on arithmetic varieties, which was proved by Moriwaki \cite{Mo1}. See Corollary \ref{intersection} and Corollary \ref{hodge index}.

\subsection*{Arithmetic Fujita Approximation}
In the geometric case, one of the most important properties of big line bundles is Fujita's approximation in \cite{Fu}. It asserts that a big line bundle can be arbitrarily closely approximated by ample line bundles. In this way, many properties of ample line bundles can be carried to big line bundles. Our arithmetic analogue is as follows:

\theoremstyle{plain}\newtheorem*{thmc}{Theorem C}
\begin{thmc}
Let $\lbb$ be a big line bundle over $X$. 
Then for any $\epsilon>0$, there exist an integer $n>0$, a birational morphism $\pi: X'\rightarrow X$ from another arithmetic variety $X'$ to $X$, and an isomorphism  
$$n\ \pi^*\lbb= \abb+ \ebb$$
for an effective line bundle $\ebb$ on $X'$ and an ample line bundle $\abb$ on $X'$
satisfying
$$\frac{1}{n^d}\vol(\abb)> \vol(\lbb)-\epsilon.$$
\end{thmc}

Our proof of this theorem consists of a finite part and an infinite part. The finite part is an analogue of Theorem 3.3 in \cite{LM}, the infinite part is solved by taking pluri-subharmonic envelope which is well-known in complex analysis.

As pointed out at the beginning, Chen \cite{Ch2} also proves this result. His proof also relies on \cite{LM}. The difference between our approaches is that, he really applies the original Theorem 3.5 of \cite{LM} combining with Bost's slope theory, while here we prove an arithmetic analogue of the theorem. Furthermore, he does not obtain our Theorem A and Theorem B.

\subsection*{Notations}
We use $\pic,\amp,\bigx,\eff$ 
to denote respectively the \textbf{isometry classes} of hermitian line bundles, ample hermitian line bundles, big hermitian line bundles, effective hermitian line bundles. 

When we treat the valuation $\nu$, we always ignore the fact that the section 0 has no image. For example, 
$\nu(S)$ is understood as $\nu(S-\{0\})$ for any subset $S\subset H^0(X,\lb)$. 

For any smooth function $f: X(\CC)\rightarrow\RR$, denote by $\obb(f)=(\ob,e^{-f})$ the trivial bundle $\ob$ endowed with the metric given by $\|1\|=e^{-f}$. In particular, it makes sense if $f=\alpha\in \RR$ is a constant function.
For any vertical cartier divisor $V$ of $X$, denote by $\obb(V)=(\ob(V),\|\cdot\|)$ the line bundle $\ob(V)$ associated to $V$ with a metric given by $\|s_V\|=1$. Here $s_V$ denotes a fixed section defining $V$.  We further denote $\lbb(f+V)=\lbb+\obb(f)+\obb(V)$ for any hermitian line bundle $\lbb$.

\

\

{\footnotesize
\noindent\textit{Acknowledgements.} 
I am indebted so much to Robert Lazarsfeld for introducing his joint work with Mircea Musta\c t\v a to me, and for his hospitality during my visit at the University of Michigan at Ann Arbor. Almost all results of this paper are based on their work. 

I would also like to thank Shou-wu Zhang for many illustrating communications. Thanks also go to S\'ebastien Boucksom for pointing out the continuity property of the envelope, to Atsushi Moriwaki for pointing out a gap in an early version of the paper, and to Yuan Yuan for clarifying many concepts in complex analysis.
}

\

\

\section{The arithmetic Okounkov body}
The main goal of this section is to prove Theorem A. Section \ref{section ample and big} recalls some results on ample line bundles and big line bundles. After considering some easy properties followed from \cite{LM}, the proof of Theorem A is reduced to Theorem \ref{comparison level m} in Section \ref{section okounkov}. Then we prove Theorem \ref{comparison level m} in the next two subsections.

\

\subsection{Basics on arithmetic ampleness and bigness}\label{section ample and big}
We follow the arithmetic intersection theory of Gillet-Soul\'e \cite{GS1} and the notion of arithmetic ampleness
by Zhang \cite{Zh}. 

Recall that an arithmetic variety is an integral scheme, projective and flat over $\Spec(\ZZ)$. The dimension means the absolute dimension. Let $X$ be an arithmetic variety of dimension $d$. The notion of hermitian line bundles needs more  words if the complex space $X(\CC)$ is not smooth. 

A metrized line bundle $\lbb=(\lb, \|\cdot\|)$ over $X$ is an invertible sheaf $\mathcal{L}$ over $X$ together with a hermitian metric $\|\cdot\|$ on each fibre of $\mathcal{L}(\CC)$ over $X(\CC)$. We say this metric is smooth if the pull-back metric over $f^*\lb$ under any analytic map $f: B^{d-1} \rightarrow X(\CC)$ is smooth in the usual sense. Here $B^{d-1}$ denotes the unit ball in $\CC^{d-1}$.
We call $\lbb$ a hermitian line bundle if its metric is smooth and invariant under complex conjugation. For a hermitian line bundle $\lbb$, we say the metric or the curvature of $\lbb$ is semipositive if the curvature of $f^*\lb$ with the pull-back metric under any analytic map $f: B^{d-1} \rightarrow X(\CC)$ is semipositive definite.

A hermitian line bundle $\lbb$ over $X$ is called ample if the following three conditions are satisfied:
\begin{itemize}
\item[(a)] $\lb_\QQ$ is ample;

\item[(b)] $\lbb$ is relatively semipositive: the curvature of $\lbb$ is semipositive and $\deg(\lb|_C) \geq 0$ for any closed curve $C$ on any special fibre of $X$ over $\mathrm{Spec}(\mathbb{Z})$;

\item[(c)] $\lbb$ is horizontally positive: the intersection number $ (\lbb|_Y)^{\dim Y}>0$ for any horizontal irreducible closed subvariety $Y$.
\end{itemize}
Zhang proved an arithmetic Nakai-Moishezon theorem which includes a special case as follows:

\begin{thm}[\cite{Zh}, Corollary 4.8]\label{zhang}
Let $\lbb$ be an ample hermitian line bundle on an arithmetic variety $X$ such that $X_\QQ$ is smooth. Then for any hermitian line bundle $\ebb$ over $X$, the $\ZZ$-module $H^0(X,\eb+N\lb)$ has a basis consisting of \textit{strictly effective sections} for $N$ large enough. 
\end{thm}

Here an effective section is a nonzero section with
supremum norm less than or equal to 1. If the
supremum norm of the section is less than 1, the section and the
line bundle are said to be strictly effective.

As for big line bundles, we need the following result.

\begin{thm}[\cite{Yu}, Theorem 2.1] \label{ample+effective}
A hermitian line bundle $\lbb$ on $X$ is big if and only if $N \lbb=\abb+\ebb$ for some integer $N>0$, some
$\abb \in \amp$ and some $\ebb\in \eff$.
\end{thm}

Fujita's approximation roughly says that we can make the ample part $\displaystyle\frac{1}{N}\abb$ arbitrarily close to $\lbb$.

In the end, we quote a theorem of Moriwaki which says that the volume function is invariant under birational morphisms. We need it when we use generic resolution of singularities.

\begin{thm}[\cite{Mo3}, Theorem 4.2]\label{moriwaki}
Let $\pi: \widetilde X\rightarrow X$ be a birational morphism of arithmetic varieties. Then for any hermitian line bundle $\lbb\in \pic$, we have
$$\vol(\pi^* \lbb)=\vol(\lbb).$$ 
\end{thm}

\

\subsection{Volumes of Okounkov bodies}\label{section okounkov}
We always assume $X$ to be normal with smooth generic fibre when we consider Okounkov bodies.
Recall that for the flag
$$X \supset Y_1 \supset Y_2 \supset \cdots \supset Y_d,$$
we require that $Y_1$ is vertical over some prime $p$ and the residue field of $Y_d$ is isomorphic to $\fp$. 
This is not essential, but we will only stick on this case for simplicity. We will first explain why such $p$ has a positive density.

Let $K$ be the largest algebraic number field contained in the fraction field of $X$. Then the structure morphism $X\rightarrow\Spec(\ZZ)$ factors through $X\rightarrow\Spec(O_K)$ where $O_K$ is the ring of integers of $K$, and $X$ is geometrically connected over $O_K$. It follows that the fiber $X_{\wp}$ over any prime ideal $\wp$ of $O_K$ is connected. It is smooth for almost all $\wp$.
We must have $Y_1=X_{\wp}$ for some $\wp$ lying over $p$. 

We require that the residue field $\fwp=O_K/\wp$ is isomorphic to $\fp$. For example, it is true if $p$ splits completely in $O_K$, and it happens with a positive density by Chebotarev's density theorem. Once this is true, 
it is easy to choose $Y_2, Y_3, \cdots, Y_{d-1}$. The existence of a point $Y_d$ follows from Weil's conjecture for curves over finite fields. 

We start with some basic properties of the Okounkov body we defined.
Recall that 
$$\nu=\nu_{Y.}=(\nu_1, \cdots, \nu_d): H^0(X,\lb)-\{0\} \rightarrow \ZZ^d$$
is the corresponding valuation map, 
$$v(\lbb)=v_{Y.}(\lbb)=\nu_{Y.}(\Hhat(X,\lbb)-\{0\})$$
is the image in $\ZZ^d$, and $\Delta=\Delta_{Y.}(\lbb)$ is the closure of $\Lambda=\Lambda_{Y.}(\lbb)=\displaystyle\bigcup_{m\geq 1}\frac{1}{m^d} v(m\lbb)$ in $\RR^d$.

\begin{lem}\label{convex and bounded}
The Okounkov body $\Delta_{Y.}(\lbb)$ is convex and bounded for any $\lbb\in\pic$.
\end{lem}

\begin{proof}
We first show convexity. By taking limit, it suffices to show that $\sum_{i=0}^k a_i x_i \in \Lambda$ for all  $x_i \in \Lambda$ and $a_i \in \QQ_{> 0}$ satisfying $\sum_{i=0}^k a_i =1$. 
Assume that $x_i$ comes from the section $s_i\in \Hhat(X,m_i\lbb)$.
Let $N$ be a positive common denominator of $a_i/m_i$, and write $a_i/m_i=b_i/N$. Then the section 
$$\bigotimes_{i=0}^k s_i^{\otimes b_i} \in \Hhat(X, N\lbb)$$
gives the point $\sum_{i=0}^k a_i x_i \in \Lambda$.

Next we show boundedness. Similar to \cite{LM}, we will show that there exists an integer $b>0$ such that 
\begin{equation} \label{bound}
\nu_i(s)\leq mb, \quad \forall \ s\in \Hhat(X, m\lbb)-\{0\}, \ i=1,\cdots, d. 
\end{equation}
The valuation $(\nu_2,\cdots, \nu_d)$ is exactly the valuation of dimension $d-1$ with respect to the flag 
$Y_1 \supset Y_2 \supset \cdots \supset Y_d$ on the ambient variety $Y_1$. Hence, the bound of $\nu_i$ for $i>1$
follows from \cite[Proposition 2.1]{LM}. It remains to bound $\nu_1$. 

Fix an ample line bundle $\abb$ on $X$. For any $s\in \Hhat(X, m\lbb)-\{0\}$, we have
$$ \lbb\cdot \abb^{d-1}=
\frac 1m  (\abb|_{\divv(s)})^{d-1}
-\frac 1m \int_{X(\CC)}\log\|s\| c_1(\mbb)^{d-1}
\geq \frac{\nu_1(s)}{m}  (\abb|_{Y_1})^{d-1}. $$
Thus we get a bound
$$\nu_1(s)\leq \frac{ \lbb\cdot \abb^{d-1}}{ (\abb|_{Y_1})^{d-1}}m. $$
\end{proof}

It is natural to describe the volume of the Okounkov body in terms of the order of the images of the valuations. The proof is actually an argument of Okounkov \cite{Ok2} using some results of Khovanskii \cite{Kh} in convex geometry, but we will only refer to the setting of \cite{LM} below.

\begin{pro}\label{volume of okounkov}
If $\lbb$ is big, then
$$\lim_{m\rightarrow \infty} \frac{\# v_{Y.}(m\lbb)}{m^d}=\vol(\Delta_{Y.}).$$
\end{pro}

\begin{proof}
Note that in our arithmetic case, 
$$\Gamma=\bigcup_{m\geq 0}(v(m\lbb),m) \subset \ZZ^{d+1}$$
is also a semigroup. We will apply \cite[Proposition 2.1]{LM} on $\Gamma$. 
We only need to check that $\Gamma$ satisfies conditions (2.3)-(2.5) required by the proposition. 
The proof is similar to Lemma 2.2 of the paper.

Condition (2.3) is trivial.
Condition (2.4) follows from (\ref{bound}) in the proof of Lemma \ref{convex and bounded}. In fact, $\Gamma$ is contained in the the semigroup generated by
$\{(x_1, x_2,\cdots, x_d, 1):  x_i=0,1,\cdots, b\}$.
It remains to check (2.5). 

We first look at the case that $\lbb$ is ample. By the arithmetic Nakai-Moishezon theorem proved by Zhang \cite{Zh} (cf. Theorem \ref{zhang}), when $m$ is sufficiently large, $H^0(m\lb)$ has a $\ZZ$-basis consisting of effective sections. By this it is easy to find an $s\in \Hhat(m\lbb)$ which is nonzero on $Y_d$, or equivalently $\nu(s)=0$. 
It follows that $(0,\cdots,0,m)\in \Gamma$. We also have $(0,\cdots,0,m+1)\in \Gamma$. Then we see that $(0,\cdots,0,1)$ is generated by two elements of $\Gamma$. It remains to show that $\bigcup_{m\geq 0} v(m\lbb)$ generates $\ZZ^d$. We will show that one $v(m\lbb)$ is enough if $m$ is sufficiently large.

For any $i=1,2,\cdots, d$, we can find a 
line bundle $\mb_i$ on $X$ with a section $t_i\in H^0(X, \mb_i)$ such that 
$t_i$ doesn't vanish on $Y_{i-1}$, vanishes on $Y_i$, and vanishes on $Y_d$ with order one. 
Then $\{\nu(t_i)\}$ is exactly the standard basis of $\ZZ^d$. Choose and fix one metric on $\mb_i$ such that $t_i$ is effective. Denote the hermitian line bundle so obtained by $\mbb_i$.
Consider the line bundle $m\lbb-\mbb_i$. We can find a section $s_i\in \Hhat(m\lbb-\mbb_i)$ with $\nu(s_i)=0$. The existence is still a simple consequence of Zhang's theorem which works on $m\lb-\mb_i$ when $m$ is large enough.
The section $s_i\otimes t_i \in \Hhat(m\lbb)$, and $\{\nu(s_i\otimes t_i)\}$ form the standard basis of $\ZZ^d$. 

Now we assume that $\lbb$ is any big line bundle. By Theorem \ref{ample+effective}, we get $N\lbb=\lbb'+\ebb$ for some integer $N>0$, some ample line bundle $\lbb'$ and some effective line bundle $\ebb$. Following the line above, we first show that $(0,\cdots,0,1)\in \ZZ^{d+1}$ is generated by $\Gamma$. Fix an effective nonzero section $e\in \ebb$. For $m$ large enough, by the above argument we have nonzero sections $s\in \Hhat(m\lbb')$ and $s'\in \Hhat(m\lbb'+\lbb)$ with valuation $\nu(s)=\nu(s')=0$. Now the sections $s\otimes e^{\otimes m}\in \Hhat(mN\lbb)$ and 
$s'\otimes e^{\otimes m}\in \Hhat((mN+1)\lbb)$ gives 
$$
(\nu(s'\otimes e^{\otimes m}),mN+1)-(\nu(s\otimes e^{\otimes m}),mN)=(0,\cdots,0,1).
$$

It remains to show that $\bigcup_{m\geq 0} v(m\lbb)$ generates $\ZZ^d$. Let $s, e$ be as above.
For any nonzero section $u\in \Hhat(m\lbb')$, we have 
$\nu(e^{\otimes m} \otimes u )-\nu(e^{\otimes m} \otimes s )=\nu(u)$
is the difference of two elements in $v(mN\lbb)$. If $m$ is sufficiently large, $v(mN\lbb)$ generates $\ZZ^d$ since $v(m\lbb')$ does.
\end{proof}

\

In the end, we state a theorem whose proof will take up the rest of this section. 
Combined with Proposition \ref{volume of okounkov}, it simply implies Theorem A and the convergence result of Chen \cite{Ch1}.

\begin{thm}\label{comparison level m}
For any $\lbb\in \pic$, there exists a constant $c=c(\lbb)$ depending only on $(X,\lbb)$, such that 
\begin{eqnarray*}
\limsup_{m\rightarrow \infty} \left|\frac{\# v_{Y.}(m\lbb)}{m^d} \log p - \frac{\hhat(X,m\lbb)}{m^d}\right|
 \leq \frac{c}{\log p}.
\end{eqnarray*}
Furthermore, we can take 
$$
c(\lbb)=2 e_0 \frac{\vol(\lb_\QQ)}{\vol(\ab_\QQ)}
\frac{\lbb\cdot \abb^{d-1}}{(d-1)!}.
$$
Here $e_0$ is the number of connected components of $X_{\overline \QQ}$, and $\abb$ is any ample line bundle on $X$.
\end{thm}

We first see how to induce the following result of Chen \cite{Ch1}.
\begin{thm}\label{limit}
The limit $\displaystyle \lim_{m\rightarrow \infty}  \frac{\hhat(X,m\lbb)}{m^d/d!}$ exists for any $\lbb\in \pic$.
\end{thm}
\begin{proof}
The case that $\lbb$ is not big is easy. Assume that $\lbb$ is big. 
The key is that $\displaystyle\frac{\# v_{Y.}(m\lbb)}{m^d}$ is convergent by Proposition \ref{volume of okounkov}. 
Then Theorem \ref{comparison level m} implies
\begin{eqnarray*}
\limsup_{m\rightarrow \infty} \frac{\hhat(X,m\lbb)}{m^d} - 
\liminf_{m\rightarrow \infty} \frac{\hhat(X,m\lbb)}{m^d} \leq \frac{2c}{\log p}.
\end{eqnarray*}
Let $p\rightarrow \infty$, we get $\limsup=\liminf$ and thus the convergence.
\end{proof}

It is also immediate to show Theorem A. In fact, since both limits exist, the result in Theorem \ref{comparison level m}
simplifies as 
$$\left|\vol(\Delta_{Y.}(\lbb))\log p-\frac{1}{d!}\vol(\lbb)\right|  \leq \frac{c}{\log p}. $$
Hence Theorem A is true under the assumption of Theorem \ref{comparison level m}.

\

\subsection{Some preliminary results}
We show some simple results which will be needed in the proof of Theorem \ref{comparison level m} in next subsection.

Let $K$ be a number field and $O_K$ be the ring of integers of $K$. 
Let $X$ be an arithmetic variety over $O_K$. In another word, the structure morphism $X\rightarrow\Spec(\ZZ)$ factors through $X\rightarrow\Spec(O_K)$. 
For any non-zero ideal $I$ of $O_K$, consider the reduction modulo $I$ map
$$r_I:H^0(X,\lb)\rightarrow H^0(X_{O_K/I},\lb_{O_K/I}). $$
We want to bound the order of $r_I(\Hhat(X,\lbb)$. 
Denote by $Z_I$ the zero locus of $I$ in $X$. Recall that the notation $\lbb(f+V)$ is explained at the end of the introduction.
The following result is a bridge from the arithmetic case to the geometric case.

\begin{pro} \label{reduction points}
For any $\lbb\in \pic$, 
\begin{eqnarray*}
\log \# r_I(\Hhat(X,\lbb))
&\leq& \hhat(X,\lbb(\log 2))-\hhat(X,\lbb(-Z_I))\\
\log \# r_I(\Hhat(X,\lbb))
&\geq& \hhat(X,\lbb)-\hhat(X, \lbb(\log 2-Z_I)).
\end{eqnarray*}
\end{pro}

\begin{proof}
For each $t\in H^0(X_{O_K/I},\lb_{O_K/I})$, fix one lifting $s_0\in r_I^{-1}(t)\cap \Hhat(X,\lbb)$ if it exists. 
For any other $s\in r_I^{-1}(t)\cap \Hhat(X,\lbb)$, we have 
$s_{Z_I}^{-1}\otimes (s-s_0)$ regular everywhere and $\|s-s_0\|_{\sup} \leq 2$. Thus
we have an element 
$$s_{Z_I}^{-1}\otimes (s-s_0) \in \Hhat(X, \lbb(\log 2-Z_I)).$$
It follows that 
$$\# (r_I^{-1}(t)\cap \Hhat(X,\lbb)) \leq \#\Hhat(X, \lbb(\log 2-Z_I)).$$
It induces the inequality
$$ \# r_I(\Hhat(X,\lbb))
\geq \frac{\#\Hhat(X,\lbb)}{\#\Hhat(X, \lbb(\log 2-Z_I))}.$$

Now we seek the upper bound of $ \# r_I(\Hhat(X,\lbb))$. Consider the set 
$$S=\Hhat(X,\lbb)+s_{Z_I}\otimes \Hhat(X,\lbb(-Z_I)).$$
Apparently $r_I(S)=r_I(\Hhat(X,\lbb))$. We further have $S\subset \Hhat(X,\lbb(\log 2))$ since 
any $s\in S$ satisfies $\|s\|_{\sup} \leq 1+1=2$.

For each $t\in r_I(S)$, there is a lifting $s_0$ of $t$ in $\Hhat(X,\lbb)$, then
$$s_0+s_{Z_I}\otimes \Hhat(X,\lbb(-Z_I))\subset r_I^{-1}(t)\cap S .$$
Hence,
$\# r_I^{-1}(t)\cap S \geq \# \Hhat(X,\lbb(-Z_I))$. It follows that
$$ \# r_I(\Hhat(X,\lbb))=\# r_I(S)
\leq \frac{\#S}{\# \Hhat(X,\lbb(-Z_I))}
\leq \frac{\#\Hhat(X,\lbb(\log 2))}{\# \Hhat(X,\lbb(-Z_I))}.$$
\end{proof}

\

\begin{lem}\label{rescaling}
For any $\lbb\in \pic$, 
$$0\leq \hhat(X,\lbb) - \hhat(X,\lbb(-\alpha)) \leq (\alpha+\log 3) \rank_\ZZ H^0(X,\lb)$$
for any $\alpha\in\RR_{+}$.
\end{lem}
\begin{proof}
The first inequality is trivial, and we only need to show the second one.
Take $I=(n)$ for an integer $n\geq 2$ in Proposition \ref{reduction points}, we get
$$ \hhat(X,\lbb)-\hhat(X, \lbb(\log 2-Z_n))
\leq \log \# r_n(\Hhat(X,\lbb)) .$$
The right-hand side has an easy bound
$$  \log \# r_n(\Hhat(X,\lbb)) \leq \log \# \left( H^0(X,\lb)/n H^0(X,\lb)\right) 
= \rank_\ZZ H^0(X,\lb) \log n .$$
It is easy to see that $\obb(Z_n)\cong\obb(\log n)$, so the above gives
$$ 
\hhat(X,\lbb)-\hhat(X, \lbb(-\log \frac n2))
\leq  \rank_\ZZ H^0(X,\lb) \log n.$$

For general $\alpha>0$, taking $n=[2e^\alpha]+1$, then the above gives
$$ 
\hhat(X,\lbb)-\hhat(X, \lbb(-\alpha))
\leq \hhat(X,\lbb)-\hhat(X, \lbb(-\log \frac n2))
\leq  \rank_\ZZ H^0(X,\lb) \log n.$$
It proves the result since 
$$
\log n \leq \log(2e^\alpha+1) \leq \alpha+\log 3.
$$

\end{proof}

\

\subsection{Comparison of the volumes}
In this subsection, we will prove Theorem \ref{comparison level m}.
Resume the notation in Section \ref{section okounkov}. That is, $K$ is the number field such that 
$X\rightarrow\Spec(O_K)$ is geometrically connected. And $Y_1=X_{\wp}$ for some prime $\wp$ of $O_K$ lying over a prime number $p$. 

Recall that $\nu=(\nu_1, \cdots, \nu_d)$ is the valuation on $X$ with respect to the flag 
$$X \supset Y_1 \supset Y_2 \supset \cdots \supset Y_d.$$
Then the flag 
$$Y_1 \supset Y_2 \supset \cdots \supset Y_d$$
 on the ambient variety $Y_1$ induces a valuation map $\nu^\circ=(\nu_2, \cdots, \nu_d)$ of dimension $d-1$ in the geometric case. They are compatible in the sense that 
$$\nu(s)=\left(\nu_1(s), \nu^\circ(
 (s_{Y_1}^{\otimes(-\nu_1(s))}s) |_{Y_1} )\right),$$ 
where $s_{Y_1}$ is the section of $\ob(Y_1)$ defining $Y_1$.
The notation such as $\lbb(\log\beta-Y_1)$ in the proposition below is explained at the end of the introduction.

\begin{pro}
For any $\lb\in \pic$,
\begin{itemize}
\item[(1)] $\displaystyle  \# \nu^\circ \left(\Hhat(X,\lbb)|_{Y_1}\right) \log p
\leq \hhat(X,\lbb(\log(2\beta)))-\hhat(X,\lbb(\log\beta-Y_1));$
\item[(2)] $\displaystyle  \# \nu^\circ \left(\Hhat(X,\lbb)|_{Y_1}\right) \log p
\geq \hhat(X,\lbb(-\log\beta))-\hhat(X,\lbb(-\log\frac{\beta}{2}-Y_1)). $
\end{itemize}
Here we denote $\beta=p\dim_{\fwp} H^0(X_{\fwp}, \lb_{\fwp})$.
\end{pro}

\begin{proof}

We first prove (1). 
The key point is to pass to the $\fwp$-subspace $\langle \Hhat(\lbb)|_{Y_1} \rangle$ of $H^0(Y_1,\lb|_{Y_1})$
generated by $\Hhat(\lbb)|_{Y_1}$. For vector spaces, we can apply the last result of \cite[Lemma 1.3]{LM} to have the order of its valuation image. It is easy to see it also works for non-algebraically closed field as long as the residue field of $Y_d$ agrees with the field of definition of the ambient variety.

We first use the trivial bound
$$\# \nu^\circ \left(\Hhat(X,\lbb)|_{Y_1}\right)  \leq \# \nu^\circ \langle \Hhat(\lbb)|_{Y_1} \rangle.$$
Then \cite[Lemma 1.3]{LM} implies
$$\# \nu^\circ \langle \Hhat(\lbb)|_{Y_1} \rangle=\dim_{\fwp} \langle \Hhat(\lbb)|_{Y_1} \rangle. $$
Thus 
$$\# \nu^\circ \left(\Hhat(X,\lbb)|_{Y_1}\right)  \leq \dim_{\fwp} \langle \Hhat(\lbb)|_{Y_1} \rangle
=\log_p\# \langle \Hhat(\lbb)|_{Y_1} \rangle . $$
Now we seek an upper bound on the order of $\langle \Hhat(\lbb)|_{Y_1} \rangle$. The key is to put this space into $\Hhat(\lbb')|_{Y_1}$ for some "bigger" hermitian line bundle $\lbb'$.

Choose a basis $\{t\}$ of $\langle \Hhat(\lbb)|_{Y_1} \rangle$ lying in $\Hhat(\lbb)|_{Y_1}$, and fix a lifting $\tilde t\in \Hhat(\lbb)$ for each $t\in \Hhat(\lbb)|_{Y_1}$. 
Since the residue field $\fwp=\fp$, the set 
$$S=\{\sum_{t}a_t \tilde t: a_t=0,1,\cdots, p-1 \}$$
maps surjectively to $\langle \Hhat(\lbb)|_{Y_1} \rangle$ under the reduction $r_\wp$.
For any such element $\sum_{t}a_t \tilde t $, the norm 
$$\|\sum_{t}a_t \tilde t\|_{\sup} \leq p\sum_{t} 1 
= p \dim_{\fwp} \langle \Hhat(\lbb)|_{Y_1} \rangle 
\leq p h^0(\lb|_{Y_1})=\beta. $$
It follows that $S \subset \Hhat (\lbb(\log\beta))$, and thus their reductions have the relation
$$\langle \Hhat(\lbb)|_{Y_1} \rangle \subset  \Hhat (\lbb(\log\beta))|_{Y_1}.$$
By this we get a bound 
$$\#\langle \Hhat(\lbb)|_{Y_1} \rangle \leq \#\Hhat (\lbb(\log\beta))|_{Y_1}.$$
By Proposition \ref{reduction points}, we obtain
$$\log \#\Hhat (\lbb(\log\beta))|_{Y_1}\leq \hhat(\lbb(\log(2\beta)))-\hhat(\lbb(\log\beta-Y_1)).$$
Putting the inequalities together, we achieve (1).

Now we prove (2). Similar to the above, we construct a set
$$T=\{\sum_{t}a_t \tilde t: a_t=0,1,\cdots, p-1 \}.$$
Here $\{t\}$ is a basis of $\langle \Hhat(\lbb(-\log\beta))|_{Y_1} \rangle$ lying in $\Hhat(\lbb(-\log\beta))|_{Y_1}$, and $\tilde t\in \Hhat(\lbb(-\log\beta))$ is a fixed lifting for each $t\in \Hhat(\lbb(-\log\beta))|_{Y_1}$. 
By the same reason, we see that 
$$T\subset \Hhat(X,\lbb)$$
and
$$\Hhat(X,\lbb)|_{Y_1} \supset T|_{Y_1}=\langle \Hhat(\lbb(-\log\beta))|_{Y_1} \rangle.$$ 
Thus we have 
\begin{align*}
\# \nu^\circ \left(\Hhat(X,\lbb)|_{Y_1}\right)
\geq  \# \nu^\circ \langle \Hhat(\lbb(-\log\beta))|_{Y_1} \rangle.
\end{align*}
By \cite[Lemma 1.3]{LM} again, we get 
\begin{align*}
\# \nu^\circ \langle \Hhat(\lbb(-\log\beta))|_{Y_1} \rangle 
= &\dim_{\fwp} \langle \Hhat(\lbb(-\log\beta))|_{Y_1} \rangle  \\
=& \log_p\# \langle \Hhat(\lbb(-\log\beta))|_{Y_1} \rangle
\geq \log_p\#  \Hhat(\lbb(-\log\beta))|_{Y_1} .  
\end{align*}
Apply Proposition \ref{reduction points} again. We have 
\begin{align*}
\# \nu^\circ \left(\Hhat(X,\lbb)|_{Y_1}\right) \log p
\geq &\log\# \Hhat(\lbb(-\log\beta))|_{Y_1}  \\
\geq &\hhat(X,\lbb(-\log\beta))-\hhat(X,\lbb(-\log\frac{\beta}{2}-Y_1)). 
\end{align*}
It proves (2).
\end{proof}

\begin{cor}
\begin{align*}
\# v(\lbb) \log p \leq &
\hhat(X,\lbb(\log(2\beta) ))+
\sum_{k\geq 1} \left( \hhat(X,\lbb(\log(2\beta)-kY_1 ))-\hhat(X,\lbb(\log\beta-kY_1))  \right),\\
\# v(\lbb) \log p \geq &
\hhat(X,\lbb(-\log\beta))-
\sum_{k\geq 1} \left( \hhat(X,\lbb(-\log\frac{\beta}{2}-kY_1 ))-\hhat(X,\lbb(-\log\beta-kY_1))  \right).
\end{align*}
\end{cor}
\begin{proof}
Denote 
$$M_k=\{s_{Y_1}^{-k}\otimes s:s\in \Hhat(X,\lbb), \nu_1(s) \geq k\}.$$
Then the compatibility between $\nu^\circ=(\nu_2, \cdots, \nu_d)$ and $\nu=(\nu_1, \nu_2, \cdots, \nu_d)$ gives
$$\# v(\lbb)=\sum_{k\geq 0} \# \nu^\circ \left(M_k|_{Y_1}\right).$$
Since the metric of $s_{Y_1}$ is identically 1 in $\obb(Y_1)$, we have an interpretation
$$M_k= \Hhat(X,\lbb(-kY_1)).$$
Therefore
$$\# v(\lbb)=\sum_{k\geq 0} \# \nu^\circ \left(\Hhat(X,\lbb(-kY_1))|_{Y_1}\right).$$
Apply the above result to each $\lbb(-kY_1)$ and rearrange the summations.

\end{proof}

\begin{remark}
The summations in both inequalities in the proposition have only finitely many nonzero terms, as we will see below.
\end{remark}

\
Now we can prove Theorem \ref{comparison level m} which asserts
\begin{eqnarray*}
\limsup_{m\rightarrow \infty} \left|\frac{\# v_{Y.}(m\lbb)}{m^d} \log p - \frac{\hhat(X,m\lbb)}{m^d}\right|
 \leq \frac{c}{\log p}.
\end{eqnarray*}

\begin{proof}
[\textbf{Proof of Theorem \ref{comparison level m}}]
The above corollary gives
\begin{align*}
\# v(m\lbb) \log p\ \leq\ &
\hhat(m\lbb+\obb(\log(2\beta_m) ))\\
& + \sum_{k\geq 1} \left( \hhat(m\lbb+\obb(\log(2\beta_m)-kY_1 ))-\hhat(m\lbb+\obb(\log\beta_m-kY_1))  \right).
\end{align*}
Here $\beta_m=p h^0(m\lb|_{Y_1})$.
By Lemma \ref{rescaling}, we get
\begin{eqnarray*}
\hhat(m\lbb+\obb(\log(2\beta_m) ))  \leq \hhat(m\lbb)+ \log(6\beta_m)\ h^0(m\lb_\QQ)
\end{eqnarray*}
and 
\begin{eqnarray*}
\hhat(m\lbb+\obb(\log(2\beta_m)-kY_1 ))-\hhat(m\lbb+\obb(\log\beta_m-kY_1))  \leq  (\log 6)\ h^0(m\lb_\QQ).
\end{eqnarray*}
Let $S$ be the set of $k\geq 1$ such that 
$$\hhat(m\lbb+\obb(\log(2\beta_m)-kY_1 ))\neq 0.$$
Then we have
\begin{eqnarray*}
\# v(m\lbb) \log p \leq 
\hhat(m\lbb)+ h^0(m\lb_\QQ)\log(6\beta_m) + (\log 6)\ h^0(m\lb_\QQ) (\# S).
\end{eqnarray*}
Next we bound $\# S$ which gives the main error term. 

Fix an ample line bundle $\abb$.
We are going to give an upper bound of $S$ in terms of intersection numbers with $\abb$.
Assume that $k\in S$, so $m\lbb+\obb(\log(2\beta_m)-kY_1 )$ is effective. We must have 
$$(m\lbb+\obb(\log(2\beta_m)-kY_1))\cdot \abb^{d-1}\geq 0.$$
Equivalently,
$$m\lbb\cdot \abb^{d-1}+\log(2\beta_m)\deg(\ab_\QQ)-kY_1\cdot \abb^{d-1}\geq 0.$$
Note that
$$Y_1\cdot \abb^{d-1}=\frac{1}{[K:\QQ]}\deg(\ab_\QQ) \log p. $$
We get 
$$
k\leq [K:\QQ] \frac{\lbb\cdot \abb^{d-1} }{\deg(\ab_\QQ) } \frac{m}{\log p}
+  [K:\QQ] \frac{\log(2\beta_m)}{\log p}. 
$$
The order of $S$ has the same bound.

Therefore,
\begin{eqnarray*}
&& \# v(m\lbb) \log p -\hhat(m\lbb) \\
&\leq &  h^0(m\lb_\QQ)\log(6\beta_m)  
 + (\log 6)\ h^0(m\lb_\QQ) [K:\QQ] \left(
\frac{\lbb\cdot \abb^{d-1} }{\deg(\ab_\QQ) } \frac{m}{\log p}
+  \frac{\log(2\beta_m)}{\log p}
\right)\\
&=& (\log 6) [K:\QQ] \frac{\vol(\lb_\QQ)}{(d-1)!} \frac{\lbb\cdot \abb^{d-1}}{\deg(\ab_\QQ)}
\frac{m^d}{\log p} + O(m^{d-1}\log m).
\end{eqnarray*}
Here we have used the fact that $\beta_m=ph^0(m\lb|_{Y_1})$ is at most a Hilbert polynomial of degree $d-1$.
It gives one direction of what we need to prove. 
Similarly, we can obtain the other direction.

\end{proof}

\section{Consequences}
In this section, $X$ is any arithmetic variety. To have good flags to apply Theorem A, we take the normalization $\widetilde X$ of the generic resolution of $X$. Consider the pull-back of hermitian line bundles. The volume does not change by Moriwaki's result quoted in Theorem \ref{moriwaki}.

\subsection{Log-concavity}

We also show the log-concavity of volume functions. The key is still the Brunn-Minkowski theorem which asserts that
$$\vol(S_1+S_2)^{\frac 1d}  \geq \vol(S_1)^{\frac 1d} + \vol(S_2)^{\frac 1d}.$$
for any two compact subsets $S_1$ and $S_2$ of $\RR^d$.
Unlike in \cite{LM}, we don't explore any universal Okounkov body since the space of numerical classes in our setting is too big.

\begin{thmb}
For any two effective line bundles $\lbb_1, \lbb_2$, we have
$$\vol(\lbb_1+\lbb_2)^{\frac 1d}\geq \vol(\lbb_1)^{\frac 1d}+ \vol(\lbb_2)^{\frac 1d}.$$
\end{thmb}

\begin{proof}
It is easy to see that the inequality is true if one of $\vol(\lbb_1)$ and $\vol(\lbb_2)$ is zero by the effectivity property. So we can assume that $\lbb_1$ and $\lbb_2$ are big. We can further assume that $X$ is normal with smooth generic fibre by Moriwaki's theorem of pull-back quoted in Theorem \ref{moriwaki}. 

Take a flag $Y.$ on $X$. It is easy to have 
$$\Lambda_{Y.}(\lbb_1)+\Lambda_{Y.}(\lbb_2)\subset \Lambda_{Y.}(\lbb_1+\lbb_2).$$
Taking closures in $\RR^d$, we get
$$\Delta_{Y.}(\lbb_1)+\Delta_{Y.}(\lbb_2)\subset \Delta_{Y.}(\lbb_1+\lbb_2).$$
The Brunn-Minkowski theorem gives
$$\vol(\Delta_{Y.}(\lbb_1+\lbb_2))^{\frac 1n}\geq \vol(\Delta_{Y.}(\lbb_1))^{\frac 1n}+\vol(\Delta_{Y.}(\lbb_2))^{\frac 1n}.$$
It implies the result by taking $\mathrm{char}(Y.)\rightarrow\infty$.
\end{proof}

\begin{cor} \label{intersection}
For any two ample line bundles $\lbb_1, \lbb_2$, we have
\begin{align*}
\lbb_1^{d-1} \cdot \lbb_2  
\geq & (\lbb_1^d  )^{\frac{d-1}{d}} (\lbb_2^d)^{\frac{1}{d}},\\
 (\lbb_1^{d-1} \cdot \lbb_2 )^2  
\geq & (\lbb_1^d  )  ( \lbb_1^{d-2} \cdot \lbb_2^2).
\end{align*}
\end{cor}
\begin{proof}
Note that volumes for ample line bundles are equal to top self-intersection numbers. We first show the first inequality.
Let $t$ be any positive rational number. 
Then the above theorem gives:
$$ \left( (\lbb_1+t\lbb_2)^d \right)^{\frac{1}{d}} 
\geq \left( \lbb_1^d \right)^{\frac{1}{d}}+ t\left(\lbb_2^d \right)^{\frac{1}{d}}.$$
In other words, 
$$ (\lbb_1+t\lbb_2)^d 
-\left( ( \lbb_1^d)^{\frac{1}{d}}+ t(\lbb_2^d )^{\frac{1}{d}} \right)^d
\geq 0.$$
The left-hand side is a polynomial in $t$ whose constant term is zero. The coefficient of degree one must be non-negative by considering $t\rightarrow 0$. It gives the result exactly.

Now we prove the second relation. 
Use the same trick on the first inequality. We get 
$$ \lbb_1^{d-1} \cdot (\lbb_1+t\lbb_2)
\geq (\lbb_1^d  )^{\frac{d-1}{d}} \left((\lbb_1+t\lbb_2)^d\right)^{\frac{1}{d}}.$$
It becomes
$$ ( \lbb_1^d+ t\lbb_1^{d-1} \cdot \lbb_2 )^d
\geq (\lbb_1^d  )^{d-1} (\lbb_1+t\lbb_2)^d.$$
Let $t\rightarrow 0$.
The terms of degree $\leq 2$ give
\begin{align*}
& (\lbb_1^d)^d+ d t (\lbb_1^d)^{d-1} (\lbb_1^{d-1} \cdot \lbb_2 )
+ \frac{d(d-1)}{2} t^2 (\lbb_1^d)^{d-2} (\lbb_1^{d-1} \cdot \lbb_2 )^2   \\
\geq & (\lbb_1^d  )^{d-1} \left( 
\lbb_1^d+  dt \lbb_1^{d-1} \cdot \lbb_2 + \frac{d(d-1)}{2} t^2 \lbb_1^{d-2} \cdot \lbb_2^2
\right).
\end{align*}
It turns out the terms of degree $\leq 1$ are canceled, and the degree two terms give
\begin{align*}
(\lbb_1^d)^{d-2} (\lbb_1^{d-1} \cdot \lbb_2 )^2  
\geq  (\lbb_1^d  )^{d-1}  ( \lbb_1^{d-2} \cdot \lbb_2^2).
\end{align*}
It gives the second inequality.

\end{proof}

\begin{remark}
The above method can also induce some inequalities on big line bundles. For example, if we allow $\lbb_2$ to be big in the proof of the first inequality, then $\lbb_1+t\lbb_2$ is still ample for $t$ small enough. Then we have
$$\vol(\lbb_2) \leq \frac{(\lbb_1^{d-1} \cdot \lbb_2)^d}{(\lbb_1^d)^{d-1}}$$
for any ample line bundle $\lbb_1$.
\end{remark}

It turns out that the second inequality above implies the arithmetic Hodge index theorem in codimension one proved by  Moriwaki \cite{Mo1}.
\begin{cor}\label{hodge index}
Let $\abb\in\amp, \lbb\in \pic$. If $\abb^{d-1}\cdot\lbb=0$, then $\abb^{d-2}\cdot\lbb^2\leq 0$.
\end{cor}

\begin{proof}
By replacing $\abb$ by its positive multiples, we can assume that $\bbb=\abb+\lbb$ is ample. Then $\lbb=\bbb-\abb$. 
We need to show that 
\begin{align} \label{want}
 \abb^{d-2}\cdot\lbb^2
=\abb^{d-2}\cdot\bbb^2-2 \abb^{d-1}\bbb + \abb^d\leq 0.
\end{align}
The above lemma gives
\begin{align} \label{have}
 (\abb^{d-1} \cdot \bbb )^2  
\geq  (\abb^d  )  ( \abb^{d-2} \cdot \bbb^2). 
\end{align}
The condition $\abb^{d-1}\cdot\lbb=0$ becomes
$\abb^{d-1}\cdot \bbb=\abb^d $. But it is easy to see that (\ref{want}) and (\ref{have}) are equivalent under this condition.

\end{proof}

\

\subsection{Fujita Approximation}

For any $\lbb\in\pic$, denote
$$
V_{k,n}(\lbb)=\{ s_1\otimes s_2 \otimes \cdots\otimes s_k:\ s_1,s_2,\cdots, s_k \in\Hhat(X, n\lbb).\}
$$ 
It is a subset of $\Hhat(X, kn\lbb)$. 
Then a consequence of \cite[Proposition 3.1]{LM} is the following result.

\begin{thm}\label{finite part}
Assume that $\lbb$ is big and $X$ is normal with smooth generic fibre, and let $Y.$ be a flag on $X$. Then for any $\epsilon>0$, there exists $n_0>0$ such that 
$$\lim_{k\rightarrow\infty}  
\frac{ \# v_{Y.}( V_{k,n}(\lbb))}{(nk)^d}
\geq \vol(\Delta_{Y.}(\lbb))-\epsilon, \quad \forall n>n_0.
$$
\end{thm}

This theorem serves as an analogue of \cite[Theorem 3.3]{LM} in the proof of the arithmetic Fujita approximation. It is not as neat as the direct arithmetic analogue:
$$\lim_{k\rightarrow\infty}  
\frac{ \log \#   V_{k,n}(\lbb)}{(nk)^d/d!}
\geq \vol(\lbb)-\epsilon, \quad \forall n>n_0.
$$
But we don't know whether this analogue is true.

For the arithmetic Fujita approximation theorem, we need an extra argument to take care of the archimedean part.

\begin{thm}\label{infinite part}
Let $L$ be an ample line bundle on a projective complex manifold $M$, and let $\|\cdot \|$ be any smooth hermitian metric on $L$. Then
\begin{itemize}
\item[(a)] There is a canonical positive continuous metric $\|\cdot \|'$ on $L$ such that $\|\cdot \|'\geq \|\cdot \|$ pointwise and $\|\cdot \|_{\sup}=\|\cdot \|_{\sup}'$ as norms on $H^0(M,mL)$ for any positive integer $m$.
\item[(b)] The above metric $\|\cdot \|'$ can be uniformly approximated by a sequence 
$\{\|\cdot \|_{j} \}_j$ of positive smooth hermitian metrics $\|\cdot \|_{j} \geq \|\cdot \|'$ on $L$ in the sense that 
$\displaystyle \lim_{j\rightarrow \infty}\frac{\|\cdot\|_{j}}{\|\cdot \|'} =1$ uniformly on $M$.
\end{itemize}
\end{thm}
\begin{proof}
The results are standard in complex analysis. We will take the metric $\|\cdot \|'$ to be the equilibrium metric of $\|\cdot \|$ defined as the following envelope: 
$$\|s(z)\|':=\inf\{\|s(z)\|_+: \|\cdot\|_+ {\rm\ positive\ singular\ metric\ on\ } L, \ \|\cdot\|_+ \geq \|\cdot\| \} .$$
By definition, it is positive and satisfies those two norm relations with $\|\cdot \|$ in (a). As for the continuity, see
\cite[Theorem 2.3 (1)]{Be} where the metric is proved to be $C^{1,1}$.

Now we consider (b). In fact, any positive continuous metric can be approximated uniformly by positive smooth ones.
It is essentially due to the technique of Demailly \cite{De}. See also \cite[Theorem 1]{BK} for a shorter proof. Note that monotone convergence to a continuous function is uniform over any compact space. Once the convergence is uniform, it is easy to obtain $\|\cdot \|_j\geq \|\cdot \|'$ by scalar perturbations on the metrics.
For more general results, we refer to \cite{BB}, especially Proposition 3.13 of the paper.
\end{proof}

Now we are prepared to prove our arithmetic Fujita approximation theorem.

\begin{thmc}
Let $\lbb$ be a big line bundle over $X$. 
Then for any $\epsilon>0$, there exist an integer $n>0$, a birational morphism $\pi: X'\rightarrow X$ from another arithmetic variety $X'$ to $X$, and an isomorphism  
$$n\ \pi^*\lbb= \abb+ \ebb$$
for an effective line bundle $\ebb$ on $X'$ and an ample line bundle $\abb$ on $X'$
satisfying
$$\frac{1}{n^d}\vol(\abb)> \vol(\lbb)-\epsilon.$$
\end{thmc}

\begin{proof}
We can assume that $X$ is normal with smooth generic fibre by Moriwaki's pull-back result.

For any positive integer $n$, let $\pi_n: X_n\rightarrow X$ be the blowing-up of the base ideal generated by $\Hhat(X, n\lbb)$. 
Denote by $\eb_n$ the line bundle associated to the exceptional divisor, with $e\in H^0(X_n,\eb_n)$ a section defining the exceptional divisor.
Denote by $\ab_n=n\ \pi^*\lb-\eb_n$, which has regular sections 
$\lambda(s)=\pi^* s\otimes e^{\otimes(-1)}$ for any $s\in \Hhat(X, n\lbb)$.
Furthermore, $\lambda (\Hhat(X, n\lbb))$ is base-point-free.

For any $s\in \Hhat(X, n\lbb)$, we have
$\|\pi^* s\|_{\sup}=\| s\|_{\sup}\leq 1$
under the pull-back metric on $\pi^*\lb$.
We claim that there exist metrics on $\eb_n$ and $\ab_n$ such that under these metrics
$n\ \pi^*\lb=\ab_n+\eb_n$ is isometric, $e$ is effective, and $\lambda(s)$ is effective for all $s\in \Hhat(X, n\lbb)$.

To get such metrics on $\eb_n$ and $\ab_n$, start from any metric $\|\cdot\|_0$ on $\eb_n$. It suffices to find a smooth function $f:X(\CC)\rightarrow \RR_{+}$ such that $f \|e\|_0\leq 1$ and $\|s\|/ (f \|e\|_0) \leq 1$ for all $s\in \Hhat(X, n\lbb)$. Equivalently, we need 
$$\sup_{s}\frac{\|s\|}{\|e\|_0} \leq f\leq \frac{1}{\|e\|_0}.$$
It is possible because the left-hand side is always less than or equal to the right-hand side, and the left-hand side is actually bounded everywhere.

Endowed with the above metrics, we get line hermitian bundles $\ebb_n$ and $\abb_n$ such that 
$$n\ \pi^*\lbb=\abb_n+\ebb_n.$$
Here $\ebb_n$ is effective and $\abb_n$ is base-point-free since $\lambda(\Hhat(X, n\lbb))\subset \Hhat(X, \abb_n)$.
Furthermore, each element of $V_{k,n}(\lbb)$ gives an effective section in 
$\Hhat(X, k\abb_n)$ by the same way. We claim that these sections give
\begin{equation} \label{lll}
\lim_{n\rightarrow\infty} \frac{1}{n^d}\vol(\abb_n)=  \vol(\lbb).
\end{equation} 
Once it is true, we have obtained decompositions satisfying all requirements of the theorem except that 
$\abb_n$ may not be ample, though it is base-point-free. We will make some adjustment to get ampleness, but let us first verify equation (\ref{lll}). It is a consequence of Theorem \ref{comparison level m} and Theorem \ref{finite part}.

In fact, it suffices to show that
$$
\liminf_{n\rightarrow\infty} \frac{1}{n^d}\vol(\abb_n)\geq \vol(\lbb).
$$
Replacing $\lbb$ by its tensor power if necessary, we can assume that $\Hhat(X, \lbb)$ is nonzero.
Let $p_0$ be an integer such that the base locus of $\Hhat(X, \lbb)$ has no vertical
irreducible components lying above any prime $p>p_0$. 
Let $Y.$ be a flag on $X$ with $\mathrm{char}(Y.)=p>p_0$ such that $Y_d$ is not contained in the base locus of $\Hhat(X, \lbb)$.
Then the same property is true for all $\Hhat(X, n\lbb)$ since its base locus is contained in that of $\Hhat(X, \lbb)$.
The strict transform $\pi_n^*Y.$ gives a flag on $X_n$. It follows that
$$
\# \nu_{\pi_n^*Y.}( \Hhat(X, k\abb_n))
\geq\# \nu_{\pi_n^*Y.}(\lambda(V_{k,n}(\lbb)))
=\# \nu_{Y.}( V_{k,n}(\lbb)).
$$
Divide both sides by $(nk)^d$, and set $k\rightarrow \infty$.
We get 
$$
\frac{1}{n^d} \vol( \Delta_{\pi_n^*Y.}(\abb_n))
\geq \lim_{k\rightarrow\infty}  
\frac{ \# v_{Y.}( V_{k,n}(\lbb))}{(nk)^d}.
$$
By Theorem \ref{finite part},
\begin{equation} \label{vvv}
\liminf_{n\rightarrow \infty}
\frac{1}{n^d} \vol( \Delta_{\pi_n^*Y.}(\abb_n))
\geq \vol( \Delta_{ Y.}(\lbb)).
\end{equation}

Now we want to take the limit as $p\rightarrow \infty$ to get volumes of the line bundles. 
It can be done by interchanging the orders of limits, which is possible since 
Theorem A is somehow uniform on $\abb_n$. 
In fact, Theorem \ref{comparison level m} gives
\begin{eqnarray*}
\left|\vol(\Delta_{Y.}(\lbb))\log p-\frac{1}{d!}\vol(\lbb)\right|  
&\leq & \frac{c(\lbb)}{\log p}, \\
\left|\vol(\Delta_{\pi_n^*Y.}(\abb_n))\log p-\frac{1}{d!}\vol(\abb_n)\right|  
&\leq & \frac{c(\abb_n)}{\log p}. 
\end{eqnarray*}
The constants $c(\lbb), c(\abb_n)$ can be taken as follows.
By choosing $\bbb\in\amp$, we can set
$$
c(\lbb)=2 e_0 \frac{\vol(\lb_\QQ)}{\vol(\bb_\QQ)}
\frac{\lbb\cdot \bbb^{d-1}}{(d-1)!}.
$$
The line bundle $\pi_n^*\bbb$ is not ample on $X_n$, but it is nef in the sense that 
$\pi_n^*\bbb^{d-1}\cdot \tbb\geq 0$ for any $\tbb\in \widehat{\rm Eff}(X_n)$.
It follows that $\pi_n^*\bbb$ can also be a reference line bundle in the proof of Theorem \ref{comparison level m}, and thus we can take 
$$
c(\abb_n)=2 e_0 \frac{\vol(\ab_{n,\QQ})}{\vol(\pi_n^*\bb_\QQ)}
\frac{\abb_n\cdot \pi_n^*\bbb^{d-1}}{(d-1)!}
\leq 2 e_0 \frac{\vol(n\lb_\QQ)}{\vol(\bb_\QQ)}
\frac{n\lbb\cdot \bbb^{d-1}}{(d-1)!}
=n^d c(\lbb).
$$
It follows that
\begin{eqnarray*}
\left|\frac{1}{n^d}\vol(\Delta_{\pi_n^*Y.}(\abb_n))\log p-\frac{1}{d!}\frac{\vol(\abb_n)}{n^d} \right|  
&\leq & \frac{c(\lbb)}{\log p}. 
\end{eqnarray*}
Therefore, taking $p\rightarrow \infty$ on equation (\ref{vvv}) yields
\begin{equation*} 
\liminf_{n\rightarrow \infty}
\frac{1}{n^d} \vol( \abb_n)
\geq \vol(\lbb).
\end{equation*}
It finishes proving (\ref{vvv}).

As mentioned above, the line bundle $\abb_n$ in the decomposition above is base-point-free but not ample in general. 
In order to make it ample, we are going to do a lot of modifications.
For simplicity, we write $(\pi_n, X_n, \abb_n)$ by $(\pi, X', \abb)$.
\begin{itemize}
\item We can further assume that $X'$ is normal with smooth generic fibre by pass to a generic resolution of singularity.
\item It suffices to treat the case that $\ab_\QQ$ is ample. 
In fact, we have $n_0\pi^*\lbb=\abb_0+\ebb_0$ for some $n_0>0$, $\abb_0\in \mathrm{amp}(X')$ and 
$\ebb_0\in \mathrm{eff}(X')$. 
Then consider 
$$(nN+n_0)\ \pi^*\lbb=(N\abb+\abb_0)+(N\ebb+\ebb_0).$$
Make $N>>n_0$. 
\item By the same trick above, we can further assume that $\lbb(-c)=(\lb, e^{c}\|\cdot\|)$ is base-point free for some $c>0$.
\item We can assume that $\abb$ has a positive smooth metric. 
By Theorem \ref{infinite part}, we can find a positive metric $\|\cdot\|_j$ of $\abb$ which is pointwise greater than $\|\cdot\|$  and satisfies
$$1\leq\frac{\|\cdot \|_{j,\sup}}{\|\cdot \|_{\sup}}<e^c.$$
Then the decomposition 
$$n\ \pi^*\lbb=(\ab,\|\cdot\|_j)+\ebb(\log \frac{\|\cdot\|_{j}}{\|\cdot \|} )$$
gives what we want.
\item The above $\abb$ is already ample. By the condition that the metric of $\abb$ is positive and $\Hhat(X',\abb)$ is base-point free, we need to check that $\abb$ is ample. 
Since $\ab$ is base-point free, it is automatically nef on any vertical subvareities.
We only need to check that 
$\chern(\lbb|_Z)^{\dim Z}>0$ for any horizontal irreducible
closed subvariety $Z$ of $X'$. 
We can find an $s\in \Hhat(X',\abb)$ such that $\mathrm{div}(s)$ does not contain
$Z$ by this base-point-free property. Use this section to compute intersection. We get
\begin{eqnarray*}
(\abb|_Z)^{\dim Z}
&=&(\abb|_{\mathrm{div}(s).Z})^{\dim
Z-1}-\int_{Z(\CC)}\log \|s\| \ c_1(\lbb)^{\dim Z-1}\\
&>&(\abb|_{\mathrm{div}(s).Z})^{\dim Z-1}.
\end{eqnarray*}
By writing $\mathrm{div}(s).Z$ as a positive linear combination of irreducible cycles, we reduce the problem to a smaller dimension. It means that the proof can be finished by induction on $\dim Z$.

\end{itemize}

\end{proof}

\

{\footnotesize
Address: School of Mathematics, Institute for Advanced Study, Einstein Drive, Princeton, NJ 08540

Email: yxy@math.ias.edu
}

\end{document}